\documentclass[11pt,english]{amsart}
\usepackage[T1]{fontenc}
\usepackage[latin9]{inputenc}
\usepackage{amstext}
\usepackage{amsthm}

\makeatletter
\usepackage{amsmath,amsfonts,amssymb,amsthm,epsfig}

\usepackage{color}
\usepackage[final,allcolors=blue,colorlinks=true]{hyperref}

\def\R{\mathbb R}

\def\S{\mathbb S}

\def\be{\beta}
\def\ga{\gamma}
\def\de{\delta}
\def\ep{\epsilon}
\def\ka{\kappa}

\def\var{\varphi}
\def\om{\omega}
\def\na{\nabla}
\def\Om{\Omega}  
\def\De{\Delta}      
\def\Si{\Sigma}      
\def\La{\Lambda}

\def\cal{\mathcal}
\def\wq{\infty}
\def\pa{\partial}

\def\loc{\text{\rm loc}}

\newcommand{\medint}{-\kern -,375cm\int}         
\newcommand{\medintinrigo}{-\kern -,315cm\int}
\newcommand{\wto}{\rightharpoonup}                

\newcommand\sing{\operatorname{sing}}

\numberwithin{equation}{section}
\newtheorem{theorem}{Theorem}[section]
\newtheorem*{theorem*}{Theorem}  

\newtheorem*{conclusion*}{Conclusin}

\newtheorem{corollary}[theorem]{Corollary}
\newtheorem*{corollary*}{Corollary}

\newtheorem{definition}[theorem]{Definition}

\newtheorem{lemma}[theorem]{Lemma}
\newtheorem*{lemma*}{Lemma}

\newtheorem*{notation*}{Notation}

\newtheorem{proposition}[theorem]{Proposition}
\newtheorem*{proposition*}{Proposition}
\newtheorem{remark}[theorem]{Remark}
\newtheorem*{remark*}{Remark}

\newtheorem*{example*}{Example}                
\theoremstyle{definition}

\makeatother

\usepackage{babel}
\begin{document}
	\title[]{A global regularity theory for shpere-valued fractional harmonic maps}
	
	\author[Y. He, C.-L. Xiang and G.-F. Zheng]{Yu He, Chang-Lin Xiang$^\ast$ and Gao-Feng Zheng}
	
	\address[Yu He]{School of Mathematics and Statistics, Central China Normal University, Wuhan 430079,  P. R.  China}
	\email{yu$\_$he@mails.ccnu.edu.cn}
	
	\address[Chang-Lin Xiang]{Three Gorges Mathematical Research Center, China Three Gorges University, 443002, Yichang, People's Republic of China}
	\email{changlin.xiang@ctgu.edu.cn}
	
	\address[Gao-Feng Zheng]{School of Mathematics and Statistics, Central China Normal University, Wuhan 430079,  P. R.  China}
	\email{gfzheng@ccnu.edu.cn}
	
	\thanks{*: corresponding author}
	\thanks{The corresponding author C.-L. Xiang is supported by NSFC (No. 11701045) and the NSF of Hubei province, P.R. China (No. 2024AFA061).  G.-F. Zheng is supported by the NSFC (No. 11571131). Both C.-L. Xiang  and  G.-F. Zheng are partly supported by the Open Research Fund of Key Laboratory of Nonlinear Analysis \& Applications  (Central China Normal University), Ministry of Education, P. R. China.}
	
	\begin{abstract}
		In this paper we consider sphere-valued stationary/minimizing fractional harmonic mappings introduced in recent years by several authors, especially by Millot-Pegon-Schikorra \cite{Millot-Pegon-Schikorra-2021-ARMA} and Millot-Sire \cite{Millot-Sire-15}. Based on their rich partial regularity theory,  we  establish a quantitative stratification theory for singular sets of these mappings by making use of the quantitative differentiation approach of Cheeger-Naber \cite{Cheeger-Naber-2013-CPAM}, from which a global regularity estimates follows.
	\end{abstract}
	
	\maketitle
	
	{\small
		\smallskip
		\newline
		\tableofcontents}
	\bigskip
	
	\section{Introduction and main results}
	
	\subsection{Background} In the series of seminar works \cite{DaLio-13-CV, DaLio-Riviere-11-Adv,DaLio-Riviere-11-APDE},  odd order harmonic mappings from $\R^m$ into a closed Riemannian manifold $N$ were considered by Da Lio and Rivi\`ere for the first time. That is, critical points of the energy functional
	\begin{equation}\label{eq: odd order HM}
	\frac{1}{2}\int_{\R^m}|(-\De)^\frac{m}{4}u|^2dx,\qquad u\in H^{m}(\R^m, N),
	\end{equation}
	where $m$ is an odd positive integer. One of their main results in \cite{DaLio-13-CV, DaLio-Riviere-11-Adv,DaLio-Riviere-11-APDE} is the smoothness of these harmonic mappings in critical dimensions, which thus extended the famous work of Helein \cite{Helein-1991,Helein-2002} on Harmonic mappings from surfaces. 	Since then, odd order harmonic mappings have been extended to various nonlocal settings, see e.g. \cite{DaLio-Schikorra-14-ACV,Millot-Pegon-2020-CV,Millot-Pegon-Schikorra-2021-ARMA,Millot-Sire-15,Millot-Sire-Wang-19,Millot-Sire-Yu-18,Roberts-18-CV,Schikorra-2012-CV,Schikorra-2015-CPDE} and the references therein. Due to its close connection with nonlocal Ginzburgh-Landau euqations, nonlocal minimal surfaces, fractional Allen-Cahn equations, free boundary value problems and etc, in the interesting series works \cite{Millot-Pegon-2020-CV, Millot-Pegon-Schikorra-2021-ARMA, Millot-Sire-15,  Millot-Sire-Wang-19, Millot-Sire-Yu-18} of Millot, Pegon, Sire,  Wang and Yu, the following type of  fractional harmonic mappings were considered. To be precise, let us assume throughout this note that $s \in(0,1)$ and $\Omega \subseteq \mathbb{R}^m$ is a bounded smooth open set. Following \cite{Millot-Pegon-Schikorra-2021-ARMA, Millot-Sire-15,  Millot-Sire-Wang-19},  define the fractional Dirichlet energy in $\Omega$ of a measurable map $u: \mathbb{R}^m \rightarrow \mathbb{R}^d$ by
	\begin{equation}\label{energy E_s}
		\mathcal{E}_s(u, \Omega):=\frac{\gamma_{m, s}}{4} \iint_{\left(\mathbb{R}^m \times \mathbb{R}^m\right) \backslash\left(\Omega^c \times \Omega^c\right)} \frac{|u(x)-u(y)|^2}{|x-y|^{m+2 s}} \mathrm{d} x \mathrm{d} y,
	\end{equation}
	where  $\Omega^c:=\mathbb{R}^m \backslash \Omega$ denotes the complement of $\Omega$. The normalisation constant $	\gamma_{m, s}=s 2^{2 s} \pi^{-\frac{m}{2}} \Gamma\left(\frac{m+2 s}{2}\right)/{\Gamma(1-s)} $ is such  that
	$$
	\mathcal{E}_s(u, \Omega)=\frac{1}{2} \int_{\mathbb{R}^m}\left|(-\Delta)^{\frac{s}{2}} u\right|^2 \mathrm{d} x, \qquad \forall\, u \in \mathcal{D}\left(\Omega ; \mathbb{R}^d\right) .
	$$
	In particular, in the case $m=1$ and $s=1/2$, this coincides with energy \eqref{eq: odd order HM} in 1-dimension. To define the fractional harmonic map, we first introduce the Hilbert space (see e.g. \cite[Section 2.1]{Millot-Pegon-Schikorra-2021-ARMA})
	\[\widehat{H}^{s}(\Om, \mathbb{R}^d):=\Big\{u\in L_{loc}^2(\mathbb{R}^m, \R^d): \mathcal{E}_s(u,\Om)<\wq\}\] 
with norm
\[\|u\|_{\widehat{H}^{s}(\Om, \mathbb{R}^d)}=\left(\|u\|_{L^2(\Om)}^2+\mathcal{E}_s(u,\Om)\right)^{\frac 12},\]
	and then  define, for a  closed Riemannian manifold $N\hookrightarrow \mathbb{R}^d$ (the embedding is isometric),
	\[\widehat{H}^{s}(\Om, N):=\Big\{u\in \widehat{H}^{s}(\Om, \mathbb{R}^d): u(x)\in N \text{ for a.e. } x\in \R^m \Big\}.\]
	Now we can recall the definition of $s$-harmonic mappings of \cite{Millot-Pegon-2020-CV, Millot-Pegon-Schikorra-2021-ARMA, Millot-Sire-15,  Millot-Sire-Wang-19, Millot-Sire-Yu-18}.	
	\begin{definition} We say that $u \in \widehat{H}^s(\Omega, N)$ is a weakly s-harmonic map, if
		\begin{equation}\label{weak s-harmonic map}
			\left.\frac{d}{d t}\right|_{t=0} \mathcal{E}_s\left(\pi_N(u+t \varphi), \Omega\right)=0 \quad \text { for all } \varphi \in C_0^{\infty}\left(\Omega, \mathbb{R}^d\right),
		\end{equation}		
		where $\pi_N$ denotes the smooth nearest point projection of $N$.
		
		If a weakly s-harmonic map $u \in \widehat{H}^s(\Omega, N)$ also satisfies
		\begin{equation}\label{stationary s-harmonic map}
			\left.\frac{d}{d t}\right|_{t=0} \mathcal{E}_s(u(x+t \psi(x)))=0 \quad \text { for all } \psi \in C_0^{\infty}\left(\Omega, \mathbb{R}^m\right),
		\end{equation}
		then it is called a stationary s-harmonic map.
		
		Finally, $u \in \widehat{H}^s(\Omega, N)$ is called a minimizing s-harmonic map if
		\begin{equation}\label{minimizing s-harmonic map}
			\mathcal{E}_s(u, \Omega) \leq \mathcal{E}_s(v, \Omega)
		\end{equation}
		for all $v \in \widehat{H}^s(\Omega, N)$ with $v=u$ on $\Omega \backslash K$ for some compact $K \subset \Omega$.
	\end{definition}
	
	In the series of  works \cite{Millot-Pegon-2020-CV, Millot-Pegon-Schikorra-2021-ARMA, Millot-Sire-15,  Millot-Sire-Wang-19, Millot-Sire-Yu-18}, a very rich partial regularity theory for $s$-harmonic mappings have been established. For our purpose, we merely collect part of their results concerning sphere-valued
	$s$-harmonic mappings in the below.
	\begin{theorem}  (\cite[Theorem 1.1]{Millot-Pegon-Schikorra-2021-ARMA})\label{thm: 1.1}
		If $m=1$ and $s\in [1/2,1)$, then each	weakly $s$-harmonic map $u \in \widehat{H}^s\left(\Omega ; \mathbb{S}^{d-1}\right)$ is smooth.
	\end{theorem}
In the case $0<s<1/2$, the same smoothness result also holds for minimizing $s$-harmonic maps, see the third conclusion of Theorem \ref{thm: 1.3} (see also \cite[Theorem 1.2]{Millot-Sire-Yu-18} for H\"older regularity). These results will play a fundamental role in our later argument. 	For stationary $s$-harmonic maps, the following partial regularity holds. That is, by defining the set of singular points as
	\[\sing(u)=\{x\in \Om: u \text{ is not continuous at } x\},\]
	there holds
	\begin{theorem}   (\cite[Theorem 1.2]{Millot-Pegon-Schikorra-2021-ARMA}) \label{thm: 1.2} Assume that $s \in(0,1)$ and $m>2 s$. If $u \in \widehat{H}^s\left(\Omega ; \mathbb{S}^{d-1}\right)$ is a stationary  $s$-harmonic map in $\Omega$, then $u \in C^{\infty}(\Omega \backslash \operatorname{sing}(u))$ and
		
		(1) for $s>1 / 2$ and $m \geqq 3, \operatorname{dim}_{\mathcal{H}} \operatorname{sing}(u) \leqq m-2$;
		
		(2) for $s>1 / 2$ and $m=2$, $\operatorname{sing}(u)$ is locally finite in $\Omega$;
		
		(3) for $s=1 / 2$ and $m \geqq 2, \mathcal{H}^{m-1}(\operatorname{sing}(u))=0$;
		
		(4) for $s<1 / 2$ and $m \geqq 2$, $\operatorname{dim}_{\mathcal{H}} \operatorname{sing}(u) \leqq m-1$;
		
		(5) for $s<1 / 2$ and $m=1$, $\operatorname{sing}(u)$ is locally finite in $\Omega$.
	\end{theorem}
	The above regularity result can be  improved for minimizing  $s$-harmonic maps.
	\begin{theorem}   (\cite[Theorem 1.3]{Millot-Pegon-Schikorra-2021-ARMA})  \label{thm: 1.3}  Assume that $s \in(0,1)$ and $m> 2s$. If $u \in \widehat{H}^s\left(\Omega ; \mathbb{S}^{d-1}\right)$ is a minimizing $s$-harmonic map in $\Omega$, then $u \in C^{\infty}(\Omega \backslash \operatorname{sing}(u))$ and
		
		(1) for $m \geqq 3, \operatorname{dim}_{\mathcal{H}} \operatorname{sing}(u) \leqq m-2$;
		
		(2) for $m=2, \operatorname{sing}(u)$ is locally finite in $\Omega$;
		
		(3) for $m=1, \operatorname{sing}(u)=\emptyset$ (that is, $u \in C^{\infty}(\Omega)$).
	\end{theorem}
	In  case the target sphere satisfies $d\ge 3$, Millot and Pegon \cite[Theorem 1.3]{Millot-Pegon-2020-CV} proved that the singular set of a minimizing $1/2$-harmonic map is even smaller:  $\operatorname{dim}_{\mathcal{H}} \operatorname{sing}(u) \leq m-3$. A similar result for minimizing $s$-harmonic maps with $s\in (0,1/2)$ can also be found in Millot, Sire and Yu \cite{Millot-Sire-Yu-18}.
	
	\subsection{Main results} Since regularity theory is basic for studying fractional harmonic mappings,   in this paper,  we aim to  derive volume estimates for the singular sets $\sing(u)$  by making use  of the quantitative differentiation approach of Cheeger and Naber \cite{Cheeger-Naber-2013-CPAM},  from which a global regularity estimate for $s$-harmonic mappings will follow. This approach has been applied to e.g. harmonic mappings, biharmonic mappings, varifolds and  currents, harmonic map flow, mean curvature flow and so on, see e.g. \cite{Breiner-Lamm-2015,Cheeger-H-N-2013,Cheeger-H-N-2015,Cheeger-Naber-2013-Invent,Cheeger-Naber-2013-CPAM,Naber-V-2020-JEMS,Naber-V-V-2019} and the references therein. To state our results precisely, let us first introduce some necessary notation. Given $\Lambda>0$, denote
	\begin{equation}\label{H_Lambda^s}
		\widehat{H}_\Lambda^{s}(\Om,N)=\left\{ u\in L_{\loc}^{2}(\R^{m},N):  {\cal E}_{s}(u,\Om)<\Lambda\right\}.
	\end{equation}
	To avoid confuse with the balls in $\R^{m+1}$, we use
	$$D_r(x)=\{y\in \R^m: |y-x|<r\}$$
	to denote the ball in $\R^m$ centered at $x$ with radius $r$, and simply write $D_r=D_r(0)$. Our first result reads as follows.
	\begin{theorem}[Integrability estimates]\label{thm: integrability}
		Given $\Lambda>0$, $m\ge 2$, and  assume  $u\in \widehat{H}_\Lambda^s(D_{4},\S^{d-1})$ is a stationary $s$-harmonic map if $1/2<s<1$, or  $u\in \widehat{H}_\Lambda^s(D_{4},\S^{d-1})$ is a minimizing $s$-harmonic map if $0<s\le 1/2$.
		Then for all $1\leq p<2$, there exists $C=C(s,m,\Lambda,p)$ such that
		\[
		\int_{D_1}  |\nabla u|^p\leq C\int_{D_1} r_u^{-p}<C.
		\]
	\end{theorem}
	Here $r_u$ is the regularity scale of $u$ defined as follows.
For a given function  $f:\Om\to \R$, we define the regularity scale function of $f$  by
	\begin{equation}\label{def: regularity scale function}
		r_{f}(x)=\max\left\{ 0\leq r\leq 1:\sup_{y\in D_{r}(x)}r|\na f(y)|\le1\right\},
	\end{equation}
	and define the set of points of $f$ with bad regularity scales by
	\[\mathcal{B}_r(f)=\{x\in \Om: r_f(x)< r\}.\]
Then 	Theorem \ref{thm: integrability} follows from the volume estimate below.
	\begin{theorem}\label{thm: regularity scale estimate}
		Given $\Lambda>0$,  $m\geq 2$, and  assume  $u\in \widehat{H}_\Lambda^s(D_{4},\S^{d-1})$ is a stationary $s$-harmonic map if $1/2<s<1$, or  $u\in \widehat{H}_\Lambda^s(D_{4},\S^{d-1})$ is a minimizing $s$-harmonic map if $0<s\le 1/2$. Then, for all $\eta>0$, there exists $C=C(s,m,\Lambda,\eta)$ such that
		\[
		\operatorname{Vol}(T_r(\mathcal{B}_r(u))\cap D_1)\leq C r^{2-\eta},\qquad \forall\,0<r<1,
		\]
		where $T_r(A)$ denotes the $r$ tubular neighborhood of a set $A$ in $\mathbb{R}^m$. Consequently, this implies that the Minkowski dimension of $ \sing(u)$ satisfies
		\[
		\dim_{\operatorname{Min}}  \sing(u)\leq m-2.
		\]
	\end{theorem}
	
	We remark that the above result does not hold for stationary $s$-harmonic maps if $0<s<1/2$. To see this, consider the map $u(x)=\chi_{\R^{m}_+}-\chi_{\overline{\R^{m}_-}}$ which has its origin in nonlocal minimal surface (see e.g. \cite{Caffarelli-R-S-10-CPAM,Savin-Valdinoci-14-JMPA}). It was proven that $u$ is a stationary $s$-harmonic map for $0<s<1/2$ in $\S^1$,  but with $m-1$ dimensional singular set with infinite $\cal{H}^{m-1}$-measure, see e.g. \cite[Remark 1.5]{Millot-Pegon-Schikorra-2021-ARMA}.
	
	Combining  the improved  estimate in Millot and Pegon \cite[Theorem 1.3]{Millot-Pegon-2020-CV}, we can improve the above theorems for minimizing $1/2$-harmonic maps as in the below.
	\begin{theorem}\label{thm: regularity scale estimate-2}
		Given $\Lambda>0$ and assume that  $u\in \widehat{H}_\Lambda^{1/2}(D_{4},\S^{d-1})$ is a minimizing $1/2$-harmonic map with $d\ge 3$. Then, for all $\eta>0$, there exists $C=C(s,m,N,\Lambda,\eta)$ such that
		\[
		\operatorname{Vol}(T_r(\mathcal{B}_r(u))\cap D_1)\leq C r^{3-\eta},\qquad \forall\,0<r<1,
		\]
		As a result, we have
		$ 	\dim_{\operatorname{Min}} \sing(u)\leq m-3$; furthermore,  there exists a constant $C=C(s,m,\Lambda,p)>0$ for each $1\leq p<3$ such that
	$		\int_{D_1}  |\nabla u|^p\leq C\int_{D_1} r_u^{-p}<C.$
	\end{theorem}
	
	To deduce the above regularity estimates, the key is to prove the following volume estimate concerning quantitative singular set $\mathcal{S}_{\eta,r}^k(u)$ (see Definition \ref{def: qs}).
	\begin{theorem}[Volume estimate of singular set]\label{volume estimate}
		Let $\La>0$, $k\in \{0,1,\cdots,m-1\}$,  $s\in (0,1)$ and  $m>2s$. Then, for all $\eta>0$ there exists $C=C(m,s,N,\Lambda,\eta)>0$ such that, for all stationary $s$-harmonic map $u\in \widehat{H}_\Lambda^s(D_{4},N)$ and all $0<r<1$,  we have
		\begin{equation}
			\operatorname{Vol}(T_r(\mathcal{S}_{\eta,r}^k(u))\cap D_1)\leq C r^{m-k-\eta}.
		\end{equation}
	\end{theorem}
	Note that, unlike that of Theorems \ref{thm: integrability}, \ref{thm: regularity scale estimate} and \ref{thm: regularity scale estimate-2},  this volume estimate in Theorem \ref{volume estimate} holds for all  stationary $s$-harmonic maps $u\in \widehat{H}_\Lambda^s(D_{4},\S^{d-1})$, and also for a general closed Riemannian manifold $N\hookrightarrow\R^d$, due to the fact that the monotonicity formula \eqref{monotonicity formular 1} holds for all such targets. As aforementioned, to prove this volume estimate, we will use the approach of Cheeger and Naber \cite{Cheeger-Naber-2013-CPAM}. However, different from the situations in \cite{Cheeger-Naber-2013-CPAM,Breiner-Lamm-2015}, in our case we first need to extend the mapping from $\R^m$ into the upper half space $\R^{m+1}_+$ so as to get  monotonicity formula  \eqref{monotonicity formular 1}, and then we need to defined a new type of quantitative symmetry so as to match the monotonicity property,  and finally we will  establish quantitative cone splitting principles so as to  find a useful cover for the quantitative singular  set $\mathcal{S}_{\eta,r}^k(u)$. For details of the proof of Theorem \ref{volume estimate}, see Section  \ref{sec: Quantitative stratification and volume estimates}.

Once Theorem \ref{volume estimate} is obtained, Theorems \ref{thm: integrability}, \ref{thm: regularity scale estimate}, \ref{thm: regularity scale estimate-2} will follow from  an $\ep$-regularity result (see Theorem \ref{thm: sym-to-reg} below) and the volume estimates of Theorem \ref{volume estimate}. Details are given in Section \ref{sec: regularity results}. To prove  Theorem \ref{thm: sym-to-reg}, we have  to assume that the    $s$-harmonic mapping $u$  is either stationary or minimal according to the range of $s$ so as to use the compactness results of \cite{Millot-Pegon-Schikorra-2021-ARMA}.
	
	Befor ending this section, we remark that the above regularity theorems are sharp in a sense. To see this, consider the $s$-harmonic map $u(x)=x/|x|$ for $x\in D_4\subset \R^2$ (see e.g. \cite[Remark 1.6]{Millot-Pegon-Schikorra-2021-ARMA}). A simple computation shows that our  result    is optimal in the sense $\na u\in L^p_\loc$ for all $1<p<2$ but not for $p=2$. However, note that $\na u$ is weakly $L^{2}$ integrable. We shall deduce this even more sharp regularity result together with the rectifiability of the stratified singular set in another paper, applying the much more sophisticated approach of Naber and Valtorta \cite{Naber-V-2017,Naber-V-2018} on harmonic mappings;   see also  \cite{GJXZ-2024} on the global regularity of biharmonic mappings for instance.
	
	\medskip	
	
	{\bf Notation.} Throughout the paper, we will use the following notations:
	
	$\bullet$ $\mathbb{R}_{+}^{m+1}=\{\mathbf{x}=(x,z):x\in \R^m, z>0\}$ denotes the $n+1$ dimensional open upper half space;
	
	$\bullet$ $B_r(\mathbf{x})$ denotes the open ball in $\mathbb{R}^{m+1}$  with radius $r$ centered at $\mathbf{x}=(x, z)$;
	
	$\bullet$ $B_r^{+}(\mathbf{x})$ denotes the half open ball in $\mathbb{R}_{+}^{m+1}$ of radius $r$ centered at $\mathbf{x}=(x, 0)$, and simply write $B^+_r=B^+_r(0)$;
	
	$\bullet$ $D_r(x)$ denotes the the open ball/disk in $\mathbb{R}^m$ centered at $x$, and write $D_r=D_r(0)$.
	
	For an arbitrary set $G \subset \mathbb{R}^{m+1}$, we write
	$$
	G^{+}:=G \cap \mathbb{R}_{+}^{m+1} \quad \text { and } \quad \partial^{+} G:=\partial G \cap \mathbb{R}_{+}^{m+1} .
	$$
If $G \subset \mathbb{R}_{+}^{m+1}$ is a bounded open set, we shall say that $G$ is admissible whenever
	
	$\bullet$ $\partial G$ is s Lipschitz regular;
	
	$\bullet$ the (relative) open set $\partial^0 G \subset \pa\mathbb{R}_{+}^{m+1}$ defined by
	$$
	\partial^0 G:=\left\{\mathbf{x} \subset \partial G \cap \pa\mathbb{R}_{+}^{m+1}: B_r^{+}(\mathbf{x}) \subset G \text { for some } r>0\right\},
	$$
	is non empty and has Lipschitz boundary; and then we have
	$\partial G=\partial^{+} G \cup \overline{\partial^0 G}$.
	
	$\bullet$ Finally, we identify $\mathbb{R}^m=\partial \mathbb{R}_{+}^{m+1}$; a set $A \subset \mathbb{R}^m$ is also identified with $A \times\{0\} \subset \partial \mathbb{R}_{+}^{m+1}$.
	
	\section{Classical singularity stratification}

	Since one of the main tools in this note is the Caffarelli-Silvestre extension of  \cite{Caffarelli-Silvestre-2007} (which may have originated in the probability literature \cite{MO-69}), we first introduce some spaces over  an open set $G \subseteq \mathbb{R}^{m+1}$. Following \cite{Millot-Pegon-Schikorra-2021-ARMA, Millot-Sire-15}, we define the weighted $L^2$-space
	$$
	L^2\left(G,|z|^a \mathrm{d} \mathbf{x}\right):=\left\{v \in L_{\text {loc }}^1(G):|z|^{\frac{a}{2}} v \in L^2(G)\right\}
	$$
	with $a=1-2 s$	and norm
	$$
	\|v\|_{L^2\left(G,|z|^a \mathrm{d} \mathbf{x}\right)}^2:=\int_G|z|^a|v|^2 \mathrm{d} \mathbf{x}.
	$$
	Accordingly, we introduce the weighted Sobolev space
	$$
	H^1\left(G,|z|^a \mathrm{d} \mathbf{x}\right):=\left\{v \in L^2\left(G,|z|^a \mathrm{d} \mathbf{x}\right): \nabla v \in L^2\left(G,|z|^a\mathrm{d} \mathbf{x}\right)\right\},
	$$
	normed by
	$$
	\|v\|_{H^1\left(G,|z|^a \mathrm{d} \mathbf{x}\right)}:=\|v\|_{L^2\left(G,|z|^a \mathrm{d} \mathbf{x}\right)}+\|\nabla v\|_{L^2\left(G,|z|^a \mathrm{d} \mathbf{x}\right)} .
	$$
	It follows that both $L^2\left(G,|z|^a \mathrm{d} \mathbf{x}\right)$ and $H^1\left(G,|z|^a \mathrm{d} \mathbf{x}\right)$ are separable Hilbert spaces when equipped with the scalar product induced by their respective Hilbertian norms.
	On $H^1\left(G,|z|^a \mathrm{d} \mathbf{x}\right)$, we define the weighted Dirichlet energy $\mathbf{E}_s(\cdot, G)$ by setting
	\begin{equation}\label{eq: 2.1}
		\mathbf{E}_s(v, G):=\frac{\boldsymbol{\delta}_s}{2} \int_G|z|^a|\nabla v|^2 \mathrm{d} \mathbf{x} \quad \text { with } \boldsymbol{\delta}_s:=2^{2 s-1} \frac{\Gamma(s)}{\Gamma(1-s)} \text {. }
	\end{equation}
	
	Some  properties of $H^1\left(G,|z|^a \mathrm{d} \mathbf{x}\right)$ are in order (see \cite{Millot-Pegon-Schikorra-2021-ARMA, Millot-Sire-15} for more results). For a bounded admissible open set $G \subseteq \mathbb{R}_{+}^{m+1}$, the space $L^2\left(G,|z|^a \mathrm{d} \mathbf{x}\right)$ embeds continuously into $L^\gamma(G)$ for every $1\leq\gamma<\frac{1}{1-s}$ whenever $s\in(0,1 / 2)$ by H\"{o}lder's inequality. For $s \in[1/2,1)$, we have $L^2\left(G,|z|^a \mathrm{d} \mathbf{x}\right) \hookrightarrow L^2(G)$ continuously since $a \leq 0$. In any case, it implies that
	$$
	H^1\left(G,|z|^a \mathrm{d} \mathbf{x}\right) \hookrightarrow W^{1, \gamma}(G)
	$$
	continuously for every $1<\gamma<\min \{1/(1-s), 2\}$. As a consequence, we have the compact embedding
\begin{equation}\label{eq: comp-embedding}
H^1\left(G,|z|^a \mathrm{d} \mathbf{x}\right)\hookrightarrow \hookrightarrow  L^\gamma(G), \qquad \forall\,1<\gamma<\min \{1/(1-s), 2\}.
\end{equation}	
	
	Now we define the $s$-harmonic extension of a given measurable function $u:\R^n\to \R$ to the half-space $\mathbb{R}_{+}^{m+1}$ by setting
	\begin{equation}\label{u^e}
		u^e(x, z):=\sigma_{m,s} \int_{\mathbb{R}^m} \frac{z^{2s} u(y)}{\left(|x-y|^2+z^2\right)^{\frac{m+2s}{2}}} d y,
	\end{equation}
	where $ \sigma_{m, s}:=\pi^{-\frac{m}{2}} {\Gamma\left(\frac{m+2 s}{2}\right)}/{\Gamma(s)} $ is a normalization constant.
	It follows that $u^e$ solves the equation
	\begin{equation}\label{eq: s-harmonic equation}
		\begin{cases}\operatorname{div}\left(z^a \nabla u^{\mathrm{e}}\right)=0 & \text { in } \mathbb{R}_{+}^{m+1} \\ u^e=u & \text { on } \partial \mathbb{R}_{+}^{m+1}\end{cases}
	\end{equation}
We mention  that the first equation of  \eqref{eq: s-harmonic equation} is  locally uniformly elliptic in $\R^{m+1}_+$ with smooth coefficient, and thus the unique continuation principle in Theorem 1.2 of \cite{Garofalo-Lin-86-Indiana} applies. That is, if both functions $u,v$ solves the first equation of \eqref{eq: s-harmonic equation} and $u\equiv v$ in an open subset of $\R^{m+1}_+$, then $u\equiv v$ holds in $\R^{m+1}_+$;  and consequently $u\equiv v$ holds also on the boundary $\R^m=\pa \R^{m+1}_+$.  This property will be used to study the symmetry of $u$ later.

It has been proved in \cite{Caffarelli-Silvestre-2007} that, 	for every   $u \in H^s\left(\mathbb{R}^m\right)$, there holds
		\begin{equation}\label{H^sE^s}
				\begin{aligned}
						{[u]_{H^s\left(\mathbb{R}^m\right)}^2 } & =\mathbf{E}_s\left(u^{\mathrm{e}}, \mathbb{R}_{+}^{m+1}\right) \\
						& =\inf \left\{\mathbf{E}_s\left(v, \mathbb{R}_{+}^{m+1}\right): v \in H^1\left(\mathbb{R}_{+}^{m+1},|z|^a \mathrm{d} \mathbf{x}\right), v=u \text { on } \mathbb{R}^m\right\},
					\end{aligned}
			\end{equation}
		where the function space $H^s(\R^m)$ is defined as follows: for any open subset $\Om\subset\R^m$, the function space $H^s(\Om)$ consists of all measurable functions $u\in L^2(\Om)$ which satisfies
		\[
	{[u]_{H^s(\Omega)}^2 } :=\frac{\gamma_{m, s}}{2} \iint_{\Omega \times \Omega} \frac{|u(x)-u(y)|^2}{|x-y|^{m+2 s}} \mathrm{~d} x \mathrm{~d} y<\infty. 	 \]
		If $u \in \widehat{H}^s(\Omega)$ for some open set $\Omega \subseteq \mathbb{R}^m$, the following estimates on $u^{\mathrm{e}}$ somehow extends the first equality in \eqref{H^sE^s} to the localized setting.	
	\begin{lemma}(\cite[Lemma 2.9]{Millot-Pegon-Schikorra-2021-ARMA}) Let $\Omega \subseteq \mathbb{R}^m$ be an open set. For every $u \in \widehat{H}^s(\Omega)$, the extension $u^{\mathrm{e}}$ given by \eqref{u^e} belongs to $H^1\left(G,|z|^a \mathrm{d} \mathbf{x}\right) \cap L_{\mathrm{loc}}^2\left(\overline{\mathbb{R}_{+}^{m+1}},|z|^a \mathrm{d} \mathbf{x}\right)$ for every bounded admissible open set $G \subseteq \mathbb{R}_{+}^{m+1}$ satisfying $\overline{\partial^0 G} \subseteq \Omega$. In addition, for every point $\mathbf{x}_0=\left(x_0, 0\right) \in \Omega \times\{0\}$ and $r>0$ such that $D_{3 r}\left(x_0\right) \subseteq \Omega$,
		\begin{equation} \left\|u^{\mathrm{e}}\right\|_{L^2\left(B_r^{+}\left(\mathbf{x}_0\right),|z|^a \mathrm{d} \mathbf{x}\right)}^2 \leq C\left(r^2 \mathcal{E}_s\left(u, D_{2 r}\left(x_0\right)\right)+r^{2-2 s}\|u\|_{L^2\left(D_{2 r}\left(x_0\right)\right)}^2\right),
		\end{equation}
		and
		\begin{equation}\label{eq: E_sbounded}
			\mathbf{E}_s\left(u^{\mathrm{e}}, B_r^{+}\left(\mathbf{x}_0\right)\right) \leq C_1 \mathcal{E}_s\left(u, D_{2 r}\left(x_0\right)\right),
		\end{equation}
		for a constant $C_1=C_1(m, s)$.
	\end{lemma}
	
	Consequently, there follows	
	
	\begin{corollary}\label{cor: conti. extension operator} Let $\Omega \subseteq \mathbb{R}^m$ be an open set and $G \subseteq \mathbb{R}_{+}^{m+1}$ a bounded admissible open set such that $\overline{\partial^0 G} \subseteq \Omega$. The extension operator $u \mapsto u^{\mathrm{e}}$ defines a continuous linear operator from $\widehat{H}^{s}(\Omega)$ into $H^1\left(G,|z|^a \mathrm{d} \mathbf{x}\right)$.
	\end{corollary}
	
	The next theorem concerns monotonicity formula of $s$-harmonic mappings, which  plays the most important role in the regularity theory of  $s$-harmonic mappings.
	\begin{theorem}(\cite[Proposition 2.17]{Millot-Pegon-Schikorra-2021-ARMA})\label{monotonicity formula} Let $\Omega \subseteq \mathbb{R}^m$ be a bounded open set. If $u \in \widehat{H}^s\left(\Omega ; \mathbb{R}^d\right)$ is stationary in $\Omega$, then for every $\mathbf{x}_0=\left(x_0, 0\right) \in \Omega \times\{0\}$, the normalized energy function
		$$
		r \in\left(0, \operatorname{dist}\left(x_0, \Omega^c\right)\right) \mapsto \boldsymbol{\Theta}_s\left(u^{\mathrm{e}}, \mathbf{x}_0, r\right):=\frac{1}{r^{m-2 s}} \mathbf{E}_s\left(u^{\mathrm{e}}, B_r^{+}\left(\mathbf{x}_0\right)\right)
		$$
		is nondecreasing. Moreover,
		\begin{equation}\label{monotonicity formular 1}
			\boldsymbol{\Theta}_s\left(u^{\mathrm{e}}, \mathbf{x}_0, r\right)-\boldsymbol{\Theta}_s\left(u^{\mathrm{e}}, \mathbf{x}_0, \rho\right)=\boldsymbol{\delta}_s \int_{B_r^{+}\left(\mathbf{x}_0\right) \backslash B_\rho^{+}\left(\mathbf{x}_0\right)} z^a \frac{\left|\left(\mathbf{x}-\mathbf{x}_0\right) \cdot \nabla u^{\mathrm{e}}\right|^2}{\left|\mathbf{x}-\mathbf{x}_0\right|^{m+2-2 s}} \mathrm{~d} \mathbf{x}
		\end{equation}
		for every $0<\rho<r<\operatorname{dist}\left(x_0, \Omega^c\right)$, where $\boldsymbol{\delta}_s$ is the constant defined  in \eqref{eq: 2.1}.
	\end{theorem}

	%
	
	Based on the monotonicity formula above,  the following partial Lipschitz regularity theorem was established by \cite[Theorem 5.1]{Millot-Pegon-Schikorra-2021-ARMA}.
	
	\begin{theorem}[Partial Regularity] \label{thm: partial regularity of MPS}
		There exist $\varepsilon_{1}=\varepsilon_{1}(m,s)>0$ and $\kappa_{2}=\kappa_{2}(m,s)\in(0,1)$
		such that the following holds. Let $u\in\widehat{H}^{s}\left(D_{2R};\mathbb{S}^{d-1}\right)$
		be a weakly $s$-harmonic map in $D_{2R}$ such that the function
		$r\in(0,2R-|\mathbf{x}|)\mapsto\boldsymbol{\Theta}_{s}\left(u^{\mathrm{e}},
		\mathbf{x},r\right)$
		is nondecreasing for every $\mathbf{x}\in\partial^{0}B_{2R}^{+}$.
		If
		\[ \boldsymbol{\Theta}_{s}\left(u^{\mathrm{e}},\boldsymbol{0},R\right)
		\leq\varepsilon_{1},
		\]
		then $u\in C^{0,1}\left(D_{\kappa_{2}R}\right)$ and
		\[
		R^{2}\|\nabla u\|_{L^{\infty}\left(D_{\kappa_{2}R}\right)}^{2}\leq C_{2}\boldsymbol{\Theta}_{s}\left(u^{\mathrm{e}},\boldsymbol{0},R\right)
		\]
		for a constant $C_{2}=C_{2}(m,s)$.\end{theorem}
	
	In terms of the regularity scale function (see definition \eqref{def: regularity scale function}), we have
	
	\begin{proposition} \label{prop: reg-scale} Suppose $u\in\widehat{H}^{s}(D_{4},\S^{d-1})$ is
		a stationary $s$-harmonic map. There exist $\varepsilon_{2}=\varepsilon_{2}(m,s)>0$
		and $\kappa_{2}=\kappa_{2}(m,s)\in(0,1)$ such that if ${\cal E}_{s}(u,D_{4})<\varepsilon_{2}$,
		then
		\[
		r_{u}(0)\ge\kappa_{2}.
		\]
	\end{proposition}
	\begin{proof}
		Choose $\varepsilon_{2}\le\min\{\varepsilon_{1}/C_{1},1/(C_{1}C_{2})\}$
		such that ${\cal E}_{s}(u,D_{4})<\varepsilon_{2}$. Then \eqref{eq: E_sbounded} implies
		\[
		\boldsymbol{\Theta}_{s}\left(v,\mathbf{x}_{0},2\right)\le C_{1}\varepsilon_{2}\le\varepsilon_{1}.
		\]
		Then Theorem \ref{thm: partial regularity of MPS} yields
		\[
		\kappa_{2}^{2}\|\nabla u\|_{L^{\infty}\left(D_{2\kappa_{2}}\right)}^{2}\leqq C_{2}\boldsymbol{\Theta}_{s}\left(u^{\mathrm{e}},\boldsymbol{0},2\right)
		\kappa_{2}^{2}\le C_{2}C_{1}\varepsilon_{2}\le1.
		\]
		This yields the result.
	\end{proof}
	
	Another important consequence of monotonicity formula is
	the compactness results of  \cite[Theorems 7.1, 7.2, 7.3]{Millot-Pegon-Schikorra-2021-ARMA}.
	
	\begin{theorem}\label{thm: compactness} (1) Assume that $s\in (0,1)\backslash\{1/2\}$ and $m>2s$. Let $\{u_{i}\}_{i\ge1}\subset\widehat{H}_{\Lambda}^{s}(D_{4},\S^{d-1})$
		be a sequence of uniformly bounded stationary $s$-harmonic map and
		$u_{i}\wto u$ in $\widehat{H}^{s}(D_{4},\R^n)$. Then $u$ is a stationary
		$s$-harmonic map in $D_{4}$, and for any open subset $\om\subset D_{4}$
		and every bounded admissible open set $G\subset\R_{+}^{m+1}$ satisfying
		$\bar{\om}\subset D_{4}$ and $\overline{\pa^{0}G}\subset D_{4}$,
		there hold
		\[
		\begin{aligned} & u_{i}\to u\qquad\text{strongly in }\widehat{H}^{s}(\om,\R^{d}).\\
			& u_{i}^{e}\to u^{e}\qquad\text{strongly in }H^{1}(G;\R^{d},|z|^{a}d\mathbf{x}).
		\end{aligned}
		\]
		
		(2) In the case $0<s\le 1/2$, the same compactness result also holds in case $\{u_{i}\}$ is a sequence of minimizing $s$-harmonic	maps.	Moreover, in this case the limit $u$ is also a minimizing $s$-harmonic map.
	\end{theorem}
	This compactness result implies (see \cite[Section 7.2]{Millot-Pegon-Schikorra-2021-ARMA} for details) that if  $u \in \widehat{H}^s\left(\Omega ; \mathbb{S}^{d-1}\right)$ is a stationary $s$-harmonic map for $s\neq 1/2$ or a minimizing $1/2$-harmonic map, then for every $x\in \Om$  and every sequence $r_k\to 0$, there exists a subsequence  $r^{\prime}_k\to 0$ and a map $\var:\R^m\to \S^{d-1}$ which is 0-homogeneous at the origin (see Definition \ref{def: classical symmetry} below) such that
	\[
	\begin{aligned} & u_{x,r_k^\prime}\to \var\qquad\text{strongly in }\widehat{H}^{s}(D_r), \text{  and  }\\
		& u_{x,r_k^\prime}^{e}\to \var^{e}\qquad\text{strongly in }H^{1}(B_r^+;\mathbb{R}^d,|z|^{a}d\mathbf{x})
	\end{aligned}
	\]
	for all $r>0$, where $u_{x,r}(y)= u(x+ry)$.
	
	\begin{definition} Let  $u \in \widehat{H}^s\left(\Omega ; \mathbb{R}^d\right)$ be a stationary $s$-harmonic map for $s\neq 1/2$, or a minimizing $1/2$-harmonic	maps for $s=1/2$. The above deduced map $\var$ is called a tangent map of $u$ at $x\in \Om$. \end{definition}	
	
	Now we recall the definition of $k$-symmetry (see, e.g. Cheeger and Naber \cite{Cheeger-Naber-2013-CPAM}).
	\begin{definition}[symmetry]\label{def: classical symmetry} Given a measurable map $\var: \mathbb{R}^m \rightarrow \mathbb{R}$. We say that
		
		(1)  $\var$ is  $0$-homogeneous or  $0$-symmetric with respect to point $p \in \mathbb{R}^m$ if $\var(p+\lambda v)=$ $\var(p+v)$ for all $\lambda>0$ and $v \in \mathbb{R}^m$.
		
		(2)  $\var$ is $k$-symmetric if $\var$ is $0$-symmetric with respect to the origin, and $\var$ is translation invariant with respect to a $k$-dimensional subspace $V \subset \mathbb{R}^m$, i.e.,
		$$
		\var(x+v)=\var(x) \quad \text { for all } x \in \mathbb{R}^m, v \in V.
		$$
	\end{definition}
	
	Then, for any $k\in\{0,1,\cdots,m\}$, we can  define for $s$-harmonic map the set
	\begin{equation}\label{eq: classical stratification}
		\Sigma^{k}(u)=\{x\in\Om:\text{no tangent map of }u\text{ is }(k+1)\text{-symmetric at }x\}.
	\end{equation}
	It is direct to verify that
	\[\Sigma^{0}(u) \subset \Sigma^{1}(u)\subset \cdots\subset \Sigma^{m-1}(u)\subset \Sigma^{m}(u)=\Om. \]
	Furthermore,  $x\not\in \Sigma^{m-1}(u)$ means that $u$ has a constant tangent map at $x$. This leads to the following  simple observation. Let $s\in (0,1)\backslash \{1/2\}$ and $u \in \widehat{H}^s\left(\Omega ; \mathbb{S}^{d-1}\right)$ be a stationary $s$-harmonic map; or $s=1/2$ and $u \in \widehat{H}^{1/2}\left(\Omega ; \mathbb{S}^{d-1}\right)$ be a minimizing $1/2$-harmonic map. Then we have
	$$\sing(u)= \Sigma^{m-1}(u).$$
	To see this, first suppose that $x\in \Om\backslash \Sigma^{m-1}(u)$;  that is, $u$ has a constant tangent map at $x$. Then  the compactness theorem \ref{thm: compactness} implies that $\boldsymbol{\Theta}_s(u^e, x, r)\to 0$ as $r\to 0$,  which in turn implies by the $\ep$-regularity theorem that $u$ is smooth in a neighborhood of $x$. Hence $\sing(u)\subset \Sigma^{m-1}(u)$.  On the other hand, if $u$ is smooth in a neighborhood of $x$, then surely there is a unique constant tangent map at $x$. This implies that $ \Sigma^{m-1}(u)\subset\sing(u)$.
	Therefore, in this case we deduce
	\[\Sigma^{0}(u) \subset \Sigma^{1}(u)\subset \cdots\subset \Sigma^{m-1}(u)=\sing(u). \]
	This is the so-called classical stratification of $\sing(u)$. In the next section we will use the approach of Cheeger and Naber \cite{Cheeger-Naber-2013-CPAM} to study each $\Sigma^{k}(u)$.

	\section{Quantitative stratification and volume estimates}\label{sec: Quantitative stratification and volume estimates}
	In  spirit of the idea in Cheeger and Naber \cite{Cheeger-Naber-2013-CPAM}, and also in order to combine the Caffarelli-Silvestre extension of a given mappings with the symmetry together, we define
	
	\begin{definition}[Quantitative symmetry] Fix a constant $1<\ga_0<\min\{1/(1-s), 2\}$.  Given a map $u \in \widehat{H}^{s}(\Omega, \R^d), \ep>0$ and nonnegative integer $k$, we say that $u$ is $(k, \ep)$-symmetric on $D_r(x) \subset \subset \Omega$, if there exists a  $k$-symmetric function $h\in  \widehat{H}^{s}(D_{2r}, \R^d)$ such that
		$$	
		\medint_{B_1^{+}}\left|u^e_{x,r}(\mathbf{y})
		-h^e_{0,r}(\mathbf{y}) \right|^{\gamma_0} d \mathbf{y}   =	\medint_{B_r^{+}(\mathbf{x})}\left|u^e(\mathbf{y})
		-h^e(\mathbf{y}-\mathbf{x})\right|^{\gamma_0} d \mathbf{y} \leq \ep
		$$
		where $\mathbf{x}=(x, 0)$.\end{definition}
	
	By Corollary \ref{cor: conti. extension operator} and the compact embedding \eqref{eq: comp-embedding}, the above integral is well defined.  A basic fact concerning the above  notion is the following  weak compactness of quantitatively symmetric functions.
	\begin{remark}\label{rmk: symmetry compactness}
		Suppose $\{u_i\}\subset \widehat{H}^{s}(\Omega)$ converges weakly to a mapping $v$  in $\widehat{H}^{s}(\Omega)$, $D_{2r}(x)\subset \Omega$ and  $u_i$ is $(k,\ep_i)$-symmetric on $D_r(x)$ for some $\ep_i\to 0$. Then $v$ is $k$-symmetric on $D_r(x)$.
	\end{remark}
To see this, we can assume without loss of generality that $x=0$ and $r=1$. Using the definition of quantitative symmetry, there exist a sequence of $k$-symmetric functions $h_i$ such that $$\medint_{B_1^{+}}\left|(u_i)^e(\mathbf{y})
-(h_i)^e(\mathbf{y}) \right|^{\gamma_0} d \mathbf{y}\le \ep_i\to 0.$$
 Since $u_i\wto v$ in $\widehat{H}^{s}(\Omega)$, we can assume up to a subsequence that $u_i^e\to v^e$ weakly in $H^1(B_1^+, z^ad\mathbf{x})$ and strongly in $L^{\ga_0}(B_1^+)$. 
Then $h_i^e\to v^e$ strongly in $L^{\ga_0}(B^+_1)$. 

On the other hand, since $h_i$ is $k$-symmetric, it is 0-homogeneous and translation invariant with respect to  a $k$-dimensional subspace $V_i\subset \R^m$, and so is $(h_i)^e$. We claim that this implies  $v^e$ is $k$-symmetric in $D_1$. To see this, first we note that $v^e$ is 0-homogeneous with respect to the origin in $B^+_1$ since so is each $h_i^e$. Hence, it follows from the unique continuation principle of \cite{Garofalo-Lin-86-Indiana}
 that $v^e$ is  0-homogeneous with respect to the origin in the whole upper half space $\R^{m+1}_+$. This in turn implies that $v$ is 0-homogeneous in $\R^m$. Secondly, note that by the $k$-symmetry of $h_i^e$, we infer that $v^e$ is translation invariant locally in $B^+_1$ in the sense that
 \begin{equation}\label{eq: symmetry of limit}
 v^e(\mathbf{x}+t)=v^e(\mathbf{x}),\qquad \forall\,\mathbf{x}\in B^{+}_{1/2} \text{ and } t\in V\times\{0\} \text{ with }|t|<1/10,
 \end{equation}
where $V\subset\R^m$ is a $k$-dimensional subspace. However, since $v^e(\cdot +t)$ satisfies the same equation as that of $v^e$, the unique continuation principle of \cite{Garofalo-Lin-86-Indiana} implies that \eqref{eq: symmetry of limit} holds for all $\mathbf{x}\in \R^{m+1}_+$. This further implies that $v^e$ is translation invariant with respect to $V$. Therefore, we can conclude that $v$ is $k$-symmetric. 

	Given the definition of quantitative symmetry, we can introduce a quantitative stratification for points of a function according to how much it is symmetric around those points.
	
	\begin{definition}[Quantitative stratification]\label{def: qs}
		For any map $u \in \widehat{H}^{s}(\Omega, N)$, $r,\eta>0$ and $k\in\{0,1,\cdots,m\}$, we define the $k$-th quantitative singular stratum $\mathcal{S}_{\eta,r}^k(u)\subset \Omega$ as
		$$\mathcal{S}^k_{\eta,r}(u)\equiv\Big\{x\in\Omega: u \text{ is not } (k+1,\eta)\text{-symmetric on }D_s(x) \text{ for any } r\le s\le 1\Big\}.$$
		Furthermore, we set
		\[
		\mathcal{S}^k_\eta(u):=\bigcap_{r>0}\mathcal{S}^k_{\eta,r}(u)
		\quad\text{ and }\quad
		\mathcal{S}^k(u)=\bigcup_{\eta>0}\mathcal{S}^k_\eta(u).
		\]
	\end{definition}
	
	It is then straightforward to verify by definition that,
	$$
	\text{If }\,\,k'\leq k,\,\eta'\geq\eta,\,r'\leq r, \qquad \text{then}\quad \mathcal{S}^{k'}_{\eta',r'}(u)\subseteq \mathcal{S}^k_{\eta,r}(u).
	$$
	The following remark shows that this definition of quantitative stratification is indeed a quantitative version of the classically defined one.
	
	\begin{remark} If $u\in\widehat{H}^{s}(\Om,N)$ is a stationary $s$-harmonic
		map for $s\neq 1/2$, or a minimizing $s$-harmonic
		map for $s=1/2$, then
		$$S^{k}(u)=\Si^{k}(u),\qquad \forall\,0\le k\le m,$$
where  $\Si^{k}(u)$ is defined as in \eqref{eq: classical stratification}.\end{remark}
	\begin{proof}
		Suppose $x\not\in S^{k}(u)$. Then, for each $i\ge1$, there exists
		$r_{i}>0$ such that $u$ is $(k+1,1/i)$-symmetric on $D_{r_{i}}(x)$.
		That is, there exist a $(k+1)$-symmetric function $h_{i}\in\widehat{H}^s(D_{2r_i})$ such that
		\[
		\medint_{B_{1}^{+}}|u_{x,r_{i}}^{e}-(h_{i})_{r_i}^e|^{\gamma_0}<1/i.
		\]
		If $r_{i}\to0$, then we obtain a tangent map $v$ of $u$ at $x$
		which is $(k+1)$-symmetric on $D_{1}$ from the above inequality (see Remark \ref{rmk: symmetry compactness}).
		If $r_{i}\ge\de>0$ for all $i\gg1$ for some $\de$, then the above
		inequality implies that $u_{x,\de}$ is $(k+1)$-symmetric on $D_{1}$,
		which still has the same consequence as the previous case. Hence $x\not\in\Si^{k}(u)$.
		Therefore, $\Si^{k}(u)\subset S^{k}(u)$.
		
		On the other hand, suppose $x\not\in\Si^{k}(u)$. Then there exist
		$r_{i}\to0$ such that $u_{x,r_{i}}^{e}\to v^{e}$ in $L^{2}(B_{1}^{+})$
		for some $(k+1)$-symmetric tangent map $v$. But this certainly implies
		that $x\not\in S_{\eta}^{k}(u)=\cap_{r>0}S_{\eta,r}^{k}(u)$ for any
		$\eta>0$. Hence $x\not\in S^{k}(u)$. Thus $S^{k}(u)\subset\Si^{k}(u)$.
		The proof is complete.
	\end{proof}
	
	The following two lemmata give a criterion on the quantitative symmetry of a given 	mapping.
	\begin{lemma}[Quantitative Rigidity]\label{quantitative rigidity}
		Let $u \in \widehat{H}_{\Lambda}^s\left(D_4,N\right)$ be a stationary $s$-harmonic map. Then for every $\ep>0$ and $0<\gamma<1 / 2$ there exist $\delta=\delta(\gamma, \ep,s, m, N, \Lambda)>0$ and $q=q(\gamma, \ep,s, m, N, \Lambda) \in \mathbb{N}$ such that for $r \in(0,1 / 2)$, if
		$$
		\boldsymbol{\Theta}_s(u^e,\boldsymbol{0},2 r)-\boldsymbol{\Theta}_s\left(u^e,\boldsymbol{0}, \gamma^q r\right) \leq \delta,
		$$
		then $u$ is $(0, \ep)$-symmetric on $D_{2r}$.	\end{lemma}

	\begin{proof} Assume there exist $\ep>0$ and $0<\gamma<1/2$ for which the statement is false. Again we assume that $r=1$. Then there exist a sequence of stationary $s$-harmonic maps  $u_i \in \widehat{H}_{\Lambda}^s\left(D_4,N\right)$ ($i=1,2,\ldots$) satisfying
		$$
		\boldsymbol{\Theta}_s(u_i^e,\boldsymbol{0},2 )-\boldsymbol{\Theta}_s\left(u_i^e,\boldsymbol{0}, \gamma^i \right) \leq \frac{1}{i}
		$$
		but none of $u_i$ is  $(0,\ep)$-symmetric on $D_2$. Up to a subsequence, we may assume that $u_i\rightharpoonup u$ in $\widehat{H}^s(D_4,N)$ and $u_i^e \rightharpoonup u^e$ in $H^1\left(B_2^+, |z|^ad\mathbf{x}\right)$. Then the Monotonicity formula \eqref{monotonicity formular 1} implies that
		$$
		\int_{B_{2}^+}z^a|\mathbf{x}\cdot \nabla u^e|^2d\mathbf{x}\leq C\liminf_{i\rightarrow \infty} \int_{B_2^+}z^a\frac{|\mathbf{x}\cdot \nabla u_i^e|^2}{|\mathbf{x}|^{m+2-2s}}d\mathbf{x}=0.
		$$			
		Thus, the $s$-harmonic function $u^e$ is $0$-homogeneous in $B_{2}^+$ with respect to the origin. Using the unique continuation theorem 1.2 of \cite{Garofalo-Lin-86-Indiana} we infer that $u^e$ is $0$-homogeneous in the whole upper space $\R^{m+1}_+$, which implies $u$ is 0-symmetric on $\R^m$. But then,  the strong convergence of $u_i^e\to u^e$ in $L^{\ga_0}(B_2^+)$ implies that
		$	\medint_{B_2^+}\left|u_i^e-u^e\right|^{\ga_0}<\ep	$
		for $i$ sufficiently large.
		Hence, $u_i$ is $(0,\ep)$-symmetric on $D_2$, which gives a contradiction.
	\end{proof}
	In the above proof we need to assume $u$ is a stationary $s$-harmonic mappings so as to use the monotonicity formula  \eqref{monotonicity formular 1}. The following lemma gives a quantitative geometric description on $k$-symmetry for all mappings in $\widehat{H}_{\Lambda}^s\left(D_4,N\right)$.
	
	\begin{lemma}[Quantitative cone splitting]\label{quantitative cone splitting}
		 Given constants $\eta, \tau,\Lambda>0$, there exists $\ep=\ep(s,m, N, \Lambda$, $\eta, \tau)>0$ such that the following holds. Let $u \in \widehat{H}_{\Lambda}^s\left(D_4,N\right)$, $x \in D_1$ and $0<r<1$. If
 $x \in \mathcal{S}_{\eta, r}^k(u)$ and $u$ is $(0, \ep)$-symmetric on $D_{2r}(x)$,		then there exists a $k$-dimensional affine subspace $V \subset  \mathbb{R}^{m}$ such that
		$$
		\left\{y \in D_r(x): u \text { is }(0,\ep)\text{-symmetric on } D_{2r}(y)\right\} \subset T_{\tau r}(V) \text {. }
		$$	
	\end{lemma}
	\begin{proof}	Assume without loss of generality that $x=0$ and $r=1$. We use a contradiction argument. Thus, for  given $\eta, \tau>0$, there exist a sequence $\left\{u_i\right\}$ with ${\cal E}_{s}(u_i,D_4)\le\La$ such that $0 \in \mathcal{S}_{\eta, 1}^k\left(u_i\right)$ for all $i, u_i$ is $(0,1/i)$-symmetric on $D_2$ and there exist points $\left\{x_1^i, x_2^i, \ldots, x_{k+1}^i\right\} \subset D_1$ satisfying the following two conditions:
		
		$\bullet$ $u_i$ is $(0,1/i)$-symmetric on each $D_2(x_j^i)$ for $j=1, \ldots, k+1$. That is,
		\[\int_{B^+_{2}(\mathbf{x}_j^i)} |u^e_i-h^e_{ij}|^{\gamma_0} d\mathbf{x}\le 1/i\] for some $0$-symmetric function $h_{ij}$.
		
		$\bullet$   dist$\left(x_j^i, \operatorname{span}\left\{0, x_1^i, \ldots, x_{j-1}^i\right\}\right) \geq \tau$ for all $j=1, \ldots, k+1$.
		
		After passing to a subsequence, there exists a map $u$ such that $u_i^e \to u^e$ weakly in $H^1(B_3^+,|z|^ad\mathbf{x})$ and strongly in $L^{\gamma_0}\left(B_3^+\right)$; and there exist points $\left\{x_1, \ldots, x_{k+1}\right\} \subset \overline{D}_1$ such that $u$ is $0$-symmetric on  $D_2(x_j)$ for all $ j=0,1, \ldots, k+1$. Here we write $x_0=0$. The distance relations are also preserved:  we have  dist$\left(x_j, \operatorname{span}\left\{x_0, x_1, \ldots, x_{j-1}\right\}\right) \geq \tau$ for all $j=0, \ldots, k+1$.
		
		It is now straightforward to verify that $u$ is $(k+1)$-symmetric on $D_1$. But then,  the strong convergence  $u_i^e \to u^e$ in $L^{\ga_0}(B_1^+)$  gives a contradiction to the assumption $0 \in \mathcal{S}_{\eta, 1}^k\left(u_i\right)$ for $i\gg 1$.		\end{proof}

	To continue, let us introduce the following notation.
	
	\begin{definition} For a stationary $s$-harmonic map $u \in \widehat{H}_{\Lambda}^s\left(D_4,N\right), x \in D_1$ and $0 \leq s_0<t_0<1$, denote
		$$
		\mathcal{W}_{s_0, t_0}(x, u):=\boldsymbol{\Theta}_s(u^e,\mathbf{x}, t_0)-\boldsymbol{\Theta}_s(u^e,\mathbf{x},s_0) \geq 0 .
		$$
	\end{definition}
	
	Notice that, by Monotonicity formula \eqref{monotonicity formular 1}, for $\left(s_1, t_1\right),\left(s_2, t_2\right)$ with $t_1 \leq s_2$,
	$$
	\mathcal{W}_{s_1, t_1}(x, u)+\mathcal{W}_{s_2, t_2}(x, u) \leq \mathcal{W}_{s_1, t_2}(x, u)
	$$
	with equality if $t_1=s_2$.
	Given constants $0<\gamma<1 / 2$ and $\delta>0$ and $q \in \mathbb{Z}^{+}$(these parameters will be fixed suitably in Lemma \ref{covering lemma}),  let $Q$ be the number of positive integers $j$ such that
	$$
	\mathcal{W}_{\gamma^{j+q}, \gamma^{j-1}}(x, u)>\delta \text {. }
	$$
	Then there has
	\begin{equation}\label{Q}
		Q \leq \frac{C_1(q+2)}{4^{m-2s}}\Lambda \delta^{-1},
	\end{equation}
	where $C_1$ is chosen as in \eqref{eq: E_sbounded}. To see this, just note that
	$$
	Q \de \le \sum_{j=1}^\wq \mathcal{W}_{\gamma^{j+q}, \gamma^{j-1}}(x, u) \leq (q+2)\mathcal{W}_{0,1}(x, u) \leq (q+2) \boldsymbol{\Theta}_s(u^e,\boldsymbol{0},2)\leq \frac{C_1(q+2)}{4^{m-2s}}\mathcal{E}_s\left(u, D_4\right).
	$$

	Following \cite{Cheeger-H-N-2013, Cheeger-H-N-2015, Cheeger-Naber-2013-Invent, Cheeger-Naber-2013-CPAM}, for each $x \in D_3$, we define a sequence $\left\{T_j(x)\right\}_{j \geq 1}$ with values in $\{0,1\}$ in the following manner. For each $j \in \mathbb{Z}^{+}$define
	$$
	T_j(x)= \begin{cases}
		1, \quad \text { if } \mathcal{W}_{\gamma^{j+q}, \gamma^{j-1}}(x, u)>\delta, \\
		0, \quad\text { if } \mathcal{W}_{\gamma^{j+q}, \gamma^{j-1}}(x, u)\leq\delta.
	\end{cases}
	$$
	\eqref{Q} implies that
	$$\sum_{j\ge 1}T_j(x)\le Q,\qquad \forall\,x\in D_3.$$
	That is,  there exist at most $Q$ nonzero entries in the sequence.  Thus, for each $\beta$-tuple $T^\be=\left(T_j^\beta\right)_{1 \leq j \leq \beta}$ with entries in $\{0,1\}$, by defining
	$$
	E_{T^\beta}(u)=\left\{x \in D_1 \mid T_j(x)=T_j^\beta \text { for } 1 \leq j \leq \beta\right\},
	$$
	we obtain a  decomposition of $D_1$ by at most $\binom{\beta}{Q} \leq \beta^Q$ non-empty such sets $E_{T^\beta}(u)$, even through aprior there have  $2^\be$ choices of such $\be$-tuple. This estimate plays an important role in the volume estimate below.
	
	\begin{lemma}[Covering Lemma]\label{covering lemma} There exists $c_0(m)<\infty$ such that, for each $\be\ge 1$, the set $\mathcal{S}_{\eta, \gamma^\beta}^j(u) \cap E_{T^\beta}(u)$ can be covered by at most $c_0\left(c_0 \gamma^{-m}\right)^Q\left(c_0 \gamma^{-j}\right)^{\beta-Q}$ balls of radius $\gamma^\beta$.
	\end{lemma}
	\begin{proof}
		For fixed $\eta, \gamma$, let $\tau=\gamma$ and choose $\ep$ as in Lemma \ref{quantitative cone splitting}. For this $\ep, \gamma$, by Lemma \ref{quantitative rigidity} there exist $\delta>0$ and $q \in \mathbb{Z}^{+}$such that if
		$$
		\boldsymbol{\Theta}_s(u^e,\mathbf{x},2\gamma^j)-\boldsymbol{\Theta}_s(u^e,\mathbf{x},\gamma^{j+q})\leq \delta
		$$
		then $u$ is $\left(0,\ep\right)$-symmetric on $D_{2\gamma^j}(x)$. Fix this $\delta, q$ throughout the proof and define $T_j(x)$ accordingly.
		
		We now determine the covering by induction argument. For $\beta=0$, we can simply choose a minimal covering of $\mathcal{S}_{\eta, 1}^j(u) \cap D_1$ by at most $c(m)$ balls of radius $1$ with centers in $\mathcal{S}_{\eta, 1}^j(u) \cap D_1$. Suppose now the statement holds for all $\be$-tuples, and given a $\be+1$ tuple $T^{\be+1}$. Here are two simple observations. First, by definition, we have $\mathcal{S}_{\eta, \gamma^{\beta+1}}^j(u) \subset \mathcal{S}_{\eta, \gamma^\beta}^j(u)$. Next, by denoting $T^\beta$ the $\beta$ tuple obtained by dropping the last entry from $T^{\beta+1}$, we immediately get $E_{T^{\beta+1}}(u) \subset E_{T^\beta}(u)$.
		
		We determine the covering recursively. For each ball $D_{\gamma^\beta}(x)$ in the covering of $\mathcal{S}_{\eta, \gamma^\beta}^j(u)\cap E_{T^\beta}(u)$, we will take a minimal covering of $D_{\gamma^\beta}(x) \cap \mathcal{S}_{\eta, \gamma^\beta}^j(u) \cap E_{T^\beta}(u)$ by balls of radius $\gamma^{\beta+1}$ as follows.  In the case $T_\beta^\beta=1$, then  we use a simple volume argument to  bound the number of balls geometrically to get a weaker bound on the covering by
		$$
		c(m) \gamma^{-m}.
		$$
		
	In the other case $T_\beta^\beta=0$, we can do better. In this case we have both $T_\beta(x)=0$ and $T_\beta(y)=0$ for all $y\in D_{\gamma^\beta}(x)\cap E_{T^\beta}(u)$, i.e. $\mathcal{W}_{\gamma^{\beta+q}, \gamma^{\beta-1}}(y, u)\leq \delta$.  By the choice of $\delta, q$, this implies that $u$ is $\left(0,\ep\right)$-symmetric on $D_{2\gamma^\beta}(y)$.  Recall that  $x \in \mathcal{S}_{\eta, \gamma^\beta}^j(u)$. Hence we can apply Lemma \ref{quantitative cone splitting} to conclude that the set $E_{T^\beta}(u) \cap D_{\gamma^\beta}(x)$ is contained in a $\gamma^{\beta+1}$ tubular neighborhood of some $j$ dimensional plane $V$. Therefore, in this case we can cover the intersection with the stronger bound on the number of balls
		$$
		c(m) \gamma^{-j} \text {. }
		$$
		
		Given any $\beta>0$ and $E_{T^\beta}(u)$, the number of times we need to apply the weaker estimate is bounded above by $Q$. Thus, the proof is complete.
	\end{proof}

	We are now ready to prove Theorem \ref{volume estimate}.
	
	\begin{proof}[Proof of Theorem \ref{volume estimate}] Choose $\gamma<1 / 2$ such that $\gamma \leq c_0^{-2 / \eta}$, where $c_0$ is as in Lemma \ref{covering lemma}. Then $c_0^\beta \leq(\gamma^\beta)^{-\eta / 2}$ and since exponentials grow faster than polynomials,
		$$
		\beta^Q \leq c(Q) c_0^\beta \leq c(\eta, m, Q)\left(\gamma^\beta\right)^{-\eta / 2} .
		$$
		
		Since	$ \operatorname{Vol}\left(D_{\gamma^\beta}(x)\right)=\omega_m \gamma^{\beta m}$
		and $D_1$ can be decomposed into at most $\beta^Q$ sets $E_{T^\beta}(u)$ for any $\beta$, we have
		$$
		\begin{aligned}
			\operatorname{Vol}\left(T_{\gamma^\beta}\left(\mathcal{S}_{\eta, \gamma^\beta}^j\right) \cap D_1\right) & \leq \beta^{Q}\left[\left(c_0 \gamma^{-m}\right)^Q\left(c_0 \gamma^{-j}\right)^{\beta-Q}\right] \omega_m \gamma^{\beta m} \\
			& \leq c(m, Q, \eta) \beta^Q c_0^\beta\left(\gamma^\beta\right)^{m-j} \\
			& \leq c(m, Q, \eta)\left(\gamma^\beta\right)^{m-j-\eta}.
		\end{aligned}
		$$
		Thus, for any $0<r<1$, by choosing $\beta>0$ such that $\gamma^{\beta+1} \leq r<\gamma^\beta$, we deduce that
		$$
		\begin{aligned}
			\operatorname{Vol}\left(T_r\left(\mathcal{S}_{\eta, r}^j\right) \cap D_1\right) & \leq \operatorname{Vol}\left(T_{\gamma^\beta}\left(\mathcal{S}_{\eta, \gamma^\beta}^j\right) \cap D_1\right) \\
			& \leq c(m, Q, \eta)\left(\gamma^\beta\right)^{m-j-\eta} \\
			& \leq c(m, Q, \eta)\left(\gamma^{-1} r\right)^{m-j-\eta} \\
			& \leq c(m, s, \eta, N, \Lambda) r^{m-j-\eta}
		\end{aligned}
		$$
		The proof is complete.
	\end{proof}

	\section{Proof of Theorems \ref{thm: integrability}, \ref{thm: regularity scale estimate} and  \ref{thm: regularity scale estimate-2}}\label{sec: regularity results}

	We first prove the following $\ep$-regularity theorem.

	\begin{theorem}\label{thm: sym-to-reg} Given $\Lambda>0$ and  assume  $u\in \widehat{H}_\Lambda^s(D_{4},\S^{d-1})$ is a stationary $s$-harmonic map if $1/2<s<1$, or  $u\in \widehat{H}_\Lambda^s(D_{4},\S^{d-1})$ is a minimizing $s$-harmonic map if $0<s\le 1/2$. There exists
		a constant $\de(m,\La,s)>0$ such that, if $u$ is $(m-1,\de)$-symmetric on $D_{2}$, then
		\[
		r_{u}(0)\ge\ka_{2},
		\]
		where $\ka_{2}=\ka_{2}(m,s)>0$	is the constant defined as in Theorem \ref{thm: partial regularity of MPS}.\end{theorem}

	The proof relies on the following lemmata.	The first
	is an $\ep$-regularity lemma, which shows that high order
	symmetry implies regularity.
	
	\begin{lemma}[$(m,\epsilon)$-Regularity]\label{lemma: new epsilon regularity}
		Given $\Lambda>0$ and  assume that $u\in \widehat{H}_\Lambda^s(D_{4},\S^{d-1})$ is a stationary $s$-harmonic map if $1/2<s<1$, or  $u\in \widehat{H}_\Lambda^s(D_{4},\S^{d-1})$ is a minimizing $s$-harmonic map if $0<s\le 1/2$.  There exists $\ep>0$ depending only on $s,m,\La,n$ such that
		\[
		r_{u}(0)\ge\kappa_{2}
		\]
		whenever $u$ is $(m,\ep)$-symmetric on $D_2$.
		\end{lemma}
	\begin{proof}
		We only consider the case $s\in (1/2,1)$, another case can be proved similarly.  Suppose, on the contrary, that there exist a sequence of stationary
		$s$-harmonic maps $u_{k}\in\widehat{H}_{\Lambda}^{s}(D_{4},\S^{d-1})$
		such that $u_{k}$ is $(m,1/k)$-symmetric on $D_2$ and  $r_{u_{k}}(0)<\kappa_{2}$. By Theorem \ref{thm: compactness},
		we can assume that $u_k\wto u$ weakly for some stationary $s$-harmonic map $u\in\widehat{H}_{\Lambda}^{s}(D_{4},\S^{d-1})$, and  $u_{k}^e\to u^e$  strongly in $H^1(B_2^+,|z|^ad\mathbf{x})$.
		Letting $k\to \wq$ we find that $u^e$ is 0-homogeneous and translation invariant with respect to the subspace $\R^m\subset \R^{m+1}$. This implies that  $u^e\equiv \text{const.}$.  But then, by the strong convergence we know that
		$$
		\boldsymbol{\Theta}_s(u_{k}^e,\boldsymbol{0},2)\rightarrow 0\quad \text{as}\quad k\rightarrow \infty,
		$$
		which implies that $r_{u_{k}}(0)\ge\kappa_{2}$  for $k \gg 1$ by Proposition \ref{prop: reg-scale}.		We reach a contradiction.
	\end{proof}

We remark that,  in the case $0<s<1/2$, the above lemma also holds for  stationary $s$-harmonic maps, since in the argument only the compactness of $s$-harmonic maps is needed.
The second ingredient of the proof of Theorem \ref{thm: sym-to-reg} is the following symmetry self-improvement
	lemma.

	\begin{lemma}[Symmetry self-improvement]\label{lem: Symmetry self-improvement}
		Given $\Lambda>0$ and  assume that $u\in \widehat{H}_\Lambda^s(D_{4},\S^{d-1})$ is a stationary $s$-harmonic map if $1/2<s<1$, or  $u\in \widehat{H}_\Lambda^s(D_{4},\S^{d-1})$ is a minimizing $s$-harmonic map if $0<s\le 1/2$. Then, for any $\ep>0$, there
		exists $\de>0$ such that if $u$ is $(m-1,\de)$-symmetric on $D_{2}$, then $u$ is also $(m,\ep)$-symmetric
		on $D_{2}$. \end{lemma}
	\begin{proof}
		We first consider the case $s\in (1/2,1)$. 	Suppose, for some $\ep_{0}>0$ and for each $k\ge1$, there is a sequence of
		stationary $s$-harmonic maps $u_{k}\in\widehat{H}_{\Lambda}^{s}(D_{4},\S^{d-1})$
		which is $(m-1,1/k)$-symmetric but not $(m,\ep_{0})$-symmetric on $D_2$.
		By the first conclusion of Theorem \ref{thm: compactness}, we can assume that $u_k\wto u$ weakly for some stationary $s$-harmonic map $u\in\widehat{H}_{\Lambda}^{s}(D_{4},\S^{d-1})$, and  $u_{k}^e\to u^e$  strongly in $H^1(B_2^+,|z|^ad\mathbf{x})$. Then
		$u$ is $(m-1)$-symmetric but not $(m,\ep_{0})$-symmetric on $D_{2}$. Now, using  the strong unique continuation result of \cite[Theorem 1.2]{Garofalo-Lin-86-Indiana} (see also the paragraph right below the equation \eqref{eq: s-harmonic equation}), we infer that $u(x)=v(x_1)$ for some $v\in H^s(D_2,\S^{d-1})$ such that (due to the homogeneity)
		\[ v(x_1)=\begin{cases}a,& \text{ if } x_1>0\\
		b,& \text{ if } x_1<0
		\end{cases}\]
		for some $a,b\in \S^{d-1}$. It is known that  $a\neq b$ implies $[v]_{H^s((-1,1))}=+\wq$ (see the argument of \cite[Lemma 7.12]{Millot-Pegon-Schikorra-2021-ARMA}). Hence $a=b$, which means that $v$ is a constant map and so $m$-symmetric in $D_2$, again a contradiction.
			The proof is complete.
		
		In the case $s\in (0,1/2]$,  arguing  as in the previous case by using  the second conclusion of the compactness theorem \ref{thm: compactness}, and also as that of \cite[Lemma 7.13]{Millot-Pegon-Schikorra-2021-ARMA}, we can conclude that $u$ is a minimizing $s$-harmonic map in the 1-dimensional interval $(-1,1)$. Thus $u$ must be a continuous map by  conclusion (3) of Theorem \ref{thm: 1.3}.  Therefore $u$ is a constant map since $u$ is also 0-homogeneous. We reach a contradiction again!
	\end{proof}
	
	Now we can prove Theorem \ref{thm: sym-to-reg}.

	\begin{proof}[Proof of Theorem \ref{thm: sym-to-reg}] It follows from Lemma  \ref{lemma: new epsilon regularity} and  \ref{lem: Symmetry self-improvement}. The proof is complete.\end{proof}
	
	Now we are ready to prove Theorems \ref{thm: integrability}, \ref{thm: regularity scale estimate}  and \ref{thm: regularity scale estimate-2}.
	
	\begin{proof}[Proof of Theorem \ref{thm: regularity scale estimate}] If $x\in \mathcal{B}_{ r}(u)$, then by Theorem \ref{thm: sym-to-reg}, $u$ is not $(m-1,\delta)$-symmetric on $D_{2r/\kappa_2}^+(x)$. In other words, $x\in \mathcal{S}_{\eta, 2r/\kappa_2}^{m-2}$ for any $0<\eta\leq \de(s,m,\Lambda)$, the constant defined as in  Theorem \ref{thm: sym-to-reg}. Therefore, we have
		\[
		\mathcal{B}_r(u)\subset \mathcal{S}_{\eta, 2r/\kappa_2}^{m-2}, \qquad \forall\, 0<\eta\leq \de(s,m,\Lambda).
		\]
		Then Theorem \ref{volume estimate} yields
		\[
		\operatorname{Vol}\left(T_r\left(\mathcal{B}_r(u)\right) \cap D_1(x)\right) \leq \operatorname{Vol}\left(T_r\left(\mathcal{S}_{\eta, 2 r/\kappa_2}^{m-2}\right) \cap D_1(x)\right) \leq C(m,s,  \Lambda, \eta) r^{2-\eta}.
		\]		
		The proof is complete.
	\end{proof}
	\begin{proof}[Proof of Theorem \ref{thm: regularity scale estimate-2}] Totally similar to the above, and  omitted. \end{proof}
	
	\begin{proof}[Proof of theorem \ref{thm: integrability}]  This theorem follows from the simple observation that
		\[\{x\in D_1: |\na u(x)|>1/r \} \subset \{x\in D_1: r_u(x)<r \}  \]
		and the volume estimate of  Theorem \ref{thm: regularity scale estimate}. 		
	\end{proof}

\end{document}